\documentclass[12pt]{amsart}

\usepackage{amssymb,amsfonts,amsthm,amsmath}
\usepackage{graphicx}
\usepackage{enumerate}
\usepackage[letterpaper,ignoreall,margin=1in]{geometry}
\usepackage[colorlinks]{hyperref}

\usepackage{booktabs}

\newtheorem*{thm*}{Theorem}
\newtheorem{thm}{Theorem}
\newtheorem{cor}[thm]{Corollary}
\newtheorem{prop}[thm]{Proposition}

\theoremstyle{remark}

\theoremstyle{definition}

\newtheorem{lem}[thm]{Lemma}

\newcommand{\Z}{\mathbb{Z}}

\newcommand{\C}{\mathbb{C}}

\DeclareMathOperator{\GL}{GL}
\DeclareMathOperator{\Hom}{Hom}

\DeclareMathOperator*{\Sym}{Sym}

\newcommand{\defining}[1]{\textbf{#1}}

\title{Power sum decompositions of defining equations of reflection arrangements}
\author{Zach Teitler}
\email{zteitler@boisestate.edu}
\address{Zach Teitler \\
Department of Mathematics \\
1910 University Drive \\
Boise State University \\
Boise, ID 83725-1555 \\
USA}
\author{Alexander Woo}
\email{awoo@uidaho.edu}
\address{Alexander Woo \\
PO Box 441103 \\
Unversity of Idaho \\
Moscow, ID 83844-1103 \\
USA}

\date{\today}
\subjclass[2010]{Primary: 15A21, 14N15; Secondary: 20F55, 13A50, 15A69}
\keywords{Waring rank, reflection group, skew invariant, reflection arrangement}

\usepackage{amsfonts, amssymb, latexsym, youngtab}

\DeclareMathOperator{\alt}{alt}

\DeclareMathOperator*{\sgn}{sgn}

\begin{document}

\bibliographystyle{amsalpha}       

\begin{abstract}
We determine the Waring rank and a Waring decomposition of the fundamental skew invariant
of any complex reflection group whose highest degree is a regular number.
This includes all irreducible real reflection groups.
\end{abstract}

\maketitle


Given a homogeneous polynomial $f$ of degree $d$, the {\bf Waring rank} of
$f$, denoted $r(f)$, is the smallest positive integer $r$ such that
there exist linear forms $\ell_1,\ldots,\ell_r$ with
$f=\ell_1^d+\cdots+\ell_r^d$,
and a {\bf Waring decomposition} is such an expression with length $r=r(f)$.
For example, $xy = \frac{1}{4} ( (x+y)^2 - (x-y)^2 )$, so $r(xy) \leq 2$,
and
\[
  xyz = \frac{1}{24} \Big( (x+y+z)^3 - (x+y-z)^3 - (x-y+z)^3 + (x-y-z)^3 \Big),
\]
so $r(xyz) \leq 4$.
Waring ranks are notoriously difficult to determine.
Even over 160 years after Sylvester's work \cite{Sylvester:1851kx} studying this question,
Waring ranks are known
only for a few families and particular cases of examples.
For a quadratic form, the rank is equal to the rank of the associated symmetric matrix.
Ranks of binary forms have been known since the 19th century,
see \cite{Sylvester:1851kx,MR2754189}.
Ranks of plane cubics such as $xyz$ have been determined by different methods several times,
such as by \cite{MR1506892} in the 1930s, \cite{comonmour96},
and \cite{Landsberg:2009yq}.
From these results, one can see that that in fact $r(xy)=2$ and $r(xyz)=4$.

A theorem of Alexander and Hirschowitz~\cite{MR1311347} gives the rank of a
sufficiently general polynomial of degree $d$ in $n$ variables, that is,
one lying in a dense open subset of the space of polynomials.
However, it is not known how to determine whether a given form is indeed general,
nor is it known in general how to determine the rank of a given form with any reasonable efficiency.
Hence there has been a body of recent work determining the rank of certain polynomials of special interest.
Some specific ranks were determined in \cite{Landsberg:2009yq},
such as $r(x(y_1^2+\dotsb+y_m^2))=r(x(y_1^2+\dotsb+y_m^2+x^2))=2m$,
$r(x^2yz)=6$, and $r(xyzw)=8$.
A result of Ranestad and Schreyer \cite{MR2842085} gives a lower bound for rank
which turns out to be tight for monomials of the form $(x_1 \dotsm x_n)^d$;
they have rank $(d+1)^{n-1}$.
More generally, the Waring ranks of arbitrary monomials were found by
Carlini, Catalisano, and Geramita \cite{Carlini20125}:
$r(x_1^{a_1} \dotsm x_n^{a_n}) = (a_2+1)\dotsm(a_n+1)$ when $a_1 \leq \dotsb \leq a_n$.
They also found the ranks of sums of pairwise coprime monomials.
Among these examples are a few sporadic cases of polynomials having greater rank than that of a general polynomial with the same number of variables and degree;
see, for example, \cite[Section 4.1]{Carlini20125}, \cite[Remark 7.3]{Landsberg:2009yq}.

One class of polynomials of great interest, generalizing the case of
monomials, is the class of products of linear forms, possibly with multiplicities.
Geometrically, the vanishing loci of these polynomials are hyperplane
(multi)-arrangements.  Even in this case, the only products of linear
forms for which the Waring rank was previously known are monomials
and binary forms.

In this paper, we consider the product of linear forms defining
the multi-arrangement associated to a complex reflection group
satisfying the technical hypothesis
that the highest degree is a regular number
(order of a regular element; see section \ref{section: regular numbers}).
In particular, our results hold for the reflection arrangement of
any irreducible real reflection group.
For example, the symmetric group
$S_n$ acts on $\C^n$ as a reflection group with the transposition
$t_{i,j}$ acting by reflecting across the hyperplane $H_{i,j}$ defined
by $x_i - x_j = 0$.  The defining equation of the union of these
hyperplanes is $V_n = \prod_{i<j} (x_i - x_j)$, the classical
Vandermonde determinant.  This is a form of degree $\binom{n}{2}$ in
$n$ variables.
We show that $r(V_n) = (n-1)!$.

We briefly sketch our proof here for the case of the Vandermonde
determinant.  To obtain an expression for $V_n$ as a sum of powers of
linear forms, one may start naively with a general linear form $\ell$ and
skew-symmetrize the power $\ell^{\binom{n}{2}}$ to give an expression
of $V_n$ as a signed sum of powers involving $|S_n| = n!$ terms.
We improve this by choosing $\ell$ to be a regular eigenvector of a Coxeter
element of $S_n$, namely a linear form with the $n$-th roots of unity as its coefficients.
With this choice, each term in the expression of $V_n$ as a sum of $n!$ powers
turns out to be repeated $n$ times.
This reduces the number of distinct terms to $(n-1)!$,
so $r(V_n) \leq (n-1)!$.
A lower bound for $r(V_n)$ is given by combining the
Ranestad--Schreyer result with well-known facts about the invariant
theory of $S_n$, yielding $r(V_n) \geq (n-1)!$.

The same strategy works for any real or
complex reflection group whose largest degree is a regular
number.  Our main theorem is the following.
\begin{thm}
Let $W$ be a finite complex reflection group acting on $\C^n$
with degrees $d_1 \leq \dotsb \leq d_n$.
Suppose that the largest degree $d_n$ of $W$ is a regular number.
(In particular, this holds for any $W$ that is an irreducible real reflection group.)
Let $f$ be the defining equation of the reflection
multi-arrangement for $W$.
Then $r(f) = d_1 \dotsm d_{n-1} = |W|/d_n$.
\end{thm}

Additionally we show how to give explicit Waring decompositions in these cases.

We do not assume that $W$ acts essentially on $\C^n$.

Other notions of rank are of interest, including cactus rank, smoothable rank, and border rank;
see Section~\ref{section: background} for details.
The lower bound of Ranestad and Schreyer also applies to cactus rank and smoothable rank,
enabling us to determine the cactus rank and smoothable rank of the defining equation
of the reflection multi-arrangement for any finite complex reflection group.
In general, we have the following.
\begin{thm}
Let $W$ be a finite complex reflection group
with degrees $d_1 \leq \dotsb \leq d_n$,
and greatest regular number $D$.
Let $f$ be the defining equation of the reflection multi-arrangement for $W$.
Then the Waring rank $r(f)$, cactus rank $cr(f)$, and smoothable rank $sr(f)$
satisfy $d_1 \dotsm d_{n-1} = |W|/d_n = cr(f) = sr(f) \leq r(f) \leq |W|/D$.
\end{thm}
In particular, when $d_n$ is a regular number, $cr(f) = sr(f) = r(f) = |W|/d_n$.

We are not, however, able to determine the border ranks of the forms that we consider.

We discuss background both on the aforementioned theorem of Ranestad and Schreyer and on complex reflection groups
and their skew invariants in Section \ref{section: background}.
The proof of our theorem is given in Section \ref{section: main proof},
and Section \ref{section: examples} discusses specific examples of reflection groups.
It is worth noting that our proof is uniform across types and does not rely on the classification of reflection groups.

\section{Background}\label{section: background}

\subsection{The Ranestad--Schreyer lower bound}

Now we recall the theorem of Ranestad and Schreyer giving a lower
bound for the rank of a homogeneous polynomial
$f$ in $n$ variables.

Consider the polynomial ring $S=\mathbb{C}[x_1,\ldots,x_n]$.
The ring $T=\mathbb{C}[\partial_1,\ldots,\partial_n]$ acts on
$S$ by differentiation; explicitly, given nonnegative integer
vectors $\alpha,\beta\in\mathbb{N}^n$,
\[
  \mathbf{\partial}^\alpha(\mathbf{x}^\beta)
  = \left( \prod_{i=1}^n \alpha_i!\binom{\beta_i}{\alpha_i} \right) \mathbf{x}^{\beta-\alpha}
\]
if
$\beta-\alpha$ is a nonnegative vector, and
$\mathbf{\partial}^\alpha(\mathbf{x}^\beta)=0$ otherwise, with the
action extending $\mathbb{C}$-linearly.
This gives rise to the {\bf apolar pairing} between $S$ and $T$.
Letting $S_k$, $T_k$ denote the homogeneous pieces of degree $k$,
this gives a perfect pairing between $S_k$ and $T_k$ for each $k$,
and a map $S_d \otimes T_a \to S_{d-a}$ whenever $d \geq a$,
or for all $d,a$ with the understanding that $S_k = T_k = 0$ when $k<0$.
This makes $S$ a $T$-module;
it is a graded $T$-module when the grading is reversed
(that is, when each $S_k$ is given degree $-k$ rather than $k$).

Given $f\in S$, define
\[
  f^\perp = \{D\in T\mid D(f)=0\}.
\]
Then $f^\perp$ is an ideal of $T$, called the {\bf apolar ideal} of $f$.
The ring $A^f:=T/f^\perp$ is a zero-dimensional Gorenstein ring,
which is known as the {\bf apolar Artinian Gorenstein ring} for $f$.
If $f$ is a homogeneous polynomial, the ideal $f^\perp$ is homogeneous and $A^f$ is graded.
As a $T$-module, $A^f$ is isomorphic to the $T$-submodule of $S$
generated by $f$ (but note that the grading is reversed).

The well-known Apolarity Lemma \cite[Lemma 1.31]{MR1735271}
states that $f$ can be written as a sum of $s$ powers of linear forms if and only if
there is an ideal $I\subseteq f^\perp$ which is the homogeneous defining ideal of a set of $s$
distinct points in projective space.
Thus the Waring rank $r(f)$ is the least $s$ such that $f^\perp$ contains the ideal of some set of
$s$ reduced points in projective space.
Generalizing this notion, the \defining{cactus rank} (called \defining{scheme length} in \cite{MR1735271})
$cr(f)$ is the least $s$ such that $f^\perp$ contains the saturated ideal of some zero-dimensional
scheme of length $s$ in projective space.
The \defining{smoothable rank} $sr(f)$ is the least $s$ such that $f^\perp$ contains the saturated
ideal of some smoothable zero-dimensional scheme of length $s$ in projective space.
For more, see \cite{Bernardi:2012fk}.
The colorful terminology ``cactus rank'' was introduced in \cite{MR2842085},
following \cite{Buczynska:2014dq}.
Clearly $cr(f) \leq sr(f) \leq r(f)$.

A homogeneous ideal $I$ is said to be {\bf generated in degree (at most) $\delta$} if there
exist polynomials $g_1,\ldots,g_k$, all of degree at most $\delta$, such
that $I=\langle g_1,\ldots,g_k\rangle$.  The theorem of Ranestad and
Schreyer is as follows.

\begin{thm}[\cite{MR2842085}] \label{thm-ranestad-schreyer-bound}
Let $f\in S$ be a homogeneous polynomial, and suppose $f^\perp$ is
generated in degree $\delta$.  Then
\[
  cr(f) \geq \frac{\dim_\mathbb{C} A^f}{\delta}.
\]
\end{thm}

The 2006 Ph.D.\ dissertation of Max Wakefield \cite{Wakefield:2006ys}
further investigates the apolar algebras of hyperplane arrangements.
In particular, he gives an example of two arrangements with the same combinatorial
intersection lattice, only one of which has an apolar ideal that is a complete
intersection.

\subsection{Reflection groups and their invariant theory}

We now give a brief overview of the invariant theory of complex
reflection groups.
Two recent books on this subject are those of Lehrer and
Taylor~\cite{Lehrer-Taylor} and Kane~\cite{Kane}.  A classic
reference, addressing only the case of real reflection groups, is that
of Humphreys~\cite{humphries}.  The statements of this section are
originally due to Steinberg~\cite{Steinberg} and Chevalley~\cite{Chevalley}.

Let $V$ be a $\C$-vector space of dimension $n$.
A non-identity element $t\in \GL(V)$
is called a {\bf pseudo-reflection} if it fixes a hyperplane
$H_t\in V$ (so it has the eigenvalue 1 with geometric multiplicity
$n-1$) and has one eigenvalue which is not equal to $1$.
This exceptional eigenvalue must be a $k_t$-th root of unity
where $k_t$ is the order of $t$.
A {\bf complex reflection group} $W$ is a finite subgroup of $\GL (V)$ generated by
pseudo-reflections.  If $V$ can be given a basis such that the matrix representation
for every element of $W$ has only real entries, then $W$ is a {\bf real reflection
group}.  It is a necessary but not sufficient condition for $W$ to be a real reflection
group that every pseudo-reflection has order 2.
The action of $W$ is said to be \defining{essential} if the only fixed point of $W$ is the origin;
equivalently, if the intersection of the reflecting hyperplanes is the origin.
We do \emph{not} assume that $W$ acts essentially.

We now change viewpoints and
consider $W$ as an abstract group equipped with a distinguished
representation $V$; from this viewpoint $V$ is known as the {\bf
  reflection representation}.  Note that $V\cong V^*$ as representations
when $W$ is a real reflection group but not necessarily when $W$ is not.

We identify $\Sym(V)$ with the polynomial ring
$S=\C[x_1,\ldots,x_n]$ by choosing a basis $x_1,\dotsc,x_n$ of $V$.
(Note this is an unusual choice of convention; $S$ is usually identified with $\Sym(V^*)$,
but we find it more convenient for our purposes to identify $T$ with $\Sym(V^*)$
and hence $S$ with $\Sym(V)$.)
For each reflection $t \in W$, let $L_t \in V$ be an exceptional
eigenvector of $t$, an eigenvector with eigenvalue $\neq 1$.
Now let $f_W=\prod L_t$, where the product is over
all reflections in $W$.  (Note that this will usually include more
reflections than are necessary to generate $W$.)
Let $H_t \subset V^*$ be the reflecting hyperplane
of $t$ in the adjoint action of $W$ on $V^*$.
Note that $H_t \subset V^*$ is defined by the linear equation $L_t$.
Two reflections $s, t$ satisfy $H_s = H_t$ if and only if $L_s = L_t$ up to a scalar multiple,
which occurs if and only if $s$ and $t$ lie in a cyclic subgroup generated by a reflection.
Therefore, up to a scalar multiple, $f_W = \prod L_H^{k_H - 1}$,
where the product is over the hyperplanes $H$
fixed by reflections in $W$, $L_H$ is the complementary exceptional eigenvector of any
reflection $t \in W$ fixing $H$, and $k_H$ is the order of the cyclic subgroup
of reflections fixing $H$.
Thus $f_W$ is the defining equation of the reflection multi-arrangement of $W$ in $V^*$.  (Of course, $f_W$ is only defined up to multiplication by a nonzero constant, but this will be irrelevant for our purposes.)

We say that a polynomial $p\in\Sym(V)$ is {\bf skew invariant}
(for $W$) if $g\cdot p = (\det g)^{-1} p$ for every $g\in W$.
(We take the determinant of $g$ by considering it as an element of $\GL(V)$.)
The polynomial $f_W$ divides every skew invariant polynomial in $S=\Sym(V)$, and we call $f_W$
the {\bf fundamental skew invariant}~\cite[Prop.~20-1B]{Kane}~\cite[Lemma 9.10]{Lehrer-Taylor}.
Let $D = \deg f = \sum (k_H-1)$ where the sum is over all reflecting hyperplanes $H$;
this is the number of reflections in $W$.
As $f_W$ divides every skew invariant polynomial, the only skew invariant polynomial of degree less than $D$ is $0$, and every
skew invariant polynomial of degree $D$ must be a constant multiple of $f_W$.
(Note that most sources consider instead elements of $\Sym(V^*)$
on which $W$ acts by $\det$, so they regard $f_W$ as the element
of $\Sym(V^*)$ defining the reflection multi-arrangment in $V$.)

A classical theorem, proved on a case by case basis by Shephard and Todd~\cite{Shephard-Todd}
and uniformly by Chevalley~\cite{Chevalley},
states that the subring $S^W\subset S$ of $W$-invariant polynomials is
itself a polynomial ring, generated by a set of homogeneous and
algebraically independent invariants. 
These invariants are not uniquely determined, but their degrees are.
These degrees will be denoted $d_1,\ldots,d_n$, with $d_1\leq\cdots\leq d_n$;
they are known as the {\bf degrees} of the complex reflection group $W$.
The product of the degrees turns out to be the order of the group $W$;
in notation, $\prod_{i=1}^n d_i = |W|$.
The degrees satisfy $\sum_{i=1}^n (d_i-1)=D$.
Note that $1$ is a degree if and only if $W$ does not act essentially.

The fundamental skew invariant has an alternative description
as the Jacobian of the invariants;
see \cite[Prop.~21-1A]{Kane},
\cite[Thm.~9.8]{Lehrer-Taylor}.

A \textbf{coinvariant} of $W$ is a homogeneous element of $T$
which is invariant under the action of $W$ (induced by the adjoint action on $V^*$).
Let $J_W$ be the ideal generated by all positive-degree
coinvariants in $T$.
Then $J_W$ is generated as an ideal by the generators of the coinvariant ring,
and these have the same degrees
$d_1,\dotsc,d_n$ as the basic invariants in $S$.
A theorem of Steinberg~\cite{Steinberg} states that $J_W = f_W^\perp$
(see also \cite[Chapter 26]{Kane} or \cite[Lemma 9.36]{Lehrer-Taylor})
so the apolar ring
$A^{f_W}$ for the fundamental skew invariant $f_W$ is isomorphic to
the {\bf covariant ring} $T/J_W$ both as a graded ring and as a
representation of $W$.
This holds without any assumption that $W$ acts essentially.
Since $J_W$ is a complete intersection ideal generated in degrees $d_1,\dotsc,d_n$,
$\dim_{\mathbb{C}} T/J_W = \prod_{i=1}^n d_i = |W|$.
Thus we have the following corollary to Theorem \ref{thm-ranestad-schreyer-bound}.

\begin{cor}\label{cor: lower bound}
Let $W$ be a complex reflection group,
$f_W$ its fundamental skew invariant,
and $d_1\leq\cdots\leq d_n$ its degrees.
Then
\[
 r(f_W) \geq sr(f_W) = cr(f_W) = \frac{d_1 \dotsm d_n}{d_n} = d_1 \dotsm d_{n-1} = \frac{|W|}{d_n} .
\]
\end{cor}
\begin{proof}
The inequalities $r(f_W) \geq s_r(f_W) \geq cr(f_W) \geq |W|/d_n$ are immediate from the definitions
and Theorem \ref{thm-ranestad-schreyer-bound}.
Let $I_1,\dotsc,I_n$ be the invariant generators of $f_W^\perp = J_W$.
This is a zero-dimensional complete intersection.
Therefore the ideal $\langle I_1,\dotsc,I_{n-1} \rangle \subset f_W^\perp$ is a one-dimensional complete intersection,
so it defines a zero-dimensional scheme of length $d_1 \dotsm d_{n-1} = |W|/d_n$ in projective space.
This shows $cr(f_W) \leq |W|/d_n$.
Since every complete intersection is smoothable, $sr(f_W) \leq |W|/d_n$.
\end{proof}

%

\subsection{Regular numbers}\label{section: regular numbers}

A vector $v\in V$, where $V$ is the reflection representation,
is called a {\bf regular vector}
if $v$ does not lie on any reflecting hyperplane.
A group element
$g\in W$ is called a {\bf regular element} if $g$ has a regular
eigenvector.  A positive integer $d$ is called a {\bf regular number}
if there exists a regular element having a regular eigenvector of
eigenvalue $e^{2\pi i/d}$.
These definitions are originally due to Springer
\cite{Springer-regular-elements}.

The remainder of this section is not necessary for our theorem or
its proof but classifies the groups for which our technical hypothesis requiring that
the highest degree be a regular number holds.  This classification was also given
(with fewer details) by Spaltenstein~\cite[p. 305--306]{Spal95}.

A reflection group $W$ is {\bf irreducible} if the reflection representation $V$ is irreducible,
or equivalently if $W$ is not the direct product of two proper reflection subgroups.

For an \emph{irreducible} real reflection group, it was classically known (even before Springer defined the concept) that $d_n$ is a regular number.  Indeed, pick any chamber (meaning connected component of the complement of the reflection arrangement) $\mathcal{C}$ of $V$ and let $t_1,\ldots,t_n$ be $n$ reflections such that the linear forms
$L_{t_1},\ldots,L_{t_n}$ are the hyperplanes bounding $\mathcal{C}$.  Then the element $t_1t_2\cdots
t_n\in W$ is known as a {\bf Coxeter element}
(and all Coxeter elements are conjugate).
The Coxeter element always has order $d_n$, and there is a real
plane in $V$, usually known as the {\bf Coxeter plane}, on which it
acts by rotation by $2\pi/d_n$; hence it has an eigenvector of
eigenvalue $e^{2\pi i/d_n}$ in the complexification of the Coxeter
plane.  This eigenvector does not lie on a reflecting hyperplane,
because the intersection of any reflecting hyperplane with the real
Coxeter plane is a real line, so the reflecting hyperplane cannot contain the complex line
containing the nonreal eigenvector.  See~\cite[Chap. 3]{humphries} for details.


More generally,
for any irreducible complex reflection group,
Lehrer and Springer~\cite{Lehrer-Springer} gave a criterion in terms of degrees
and codegrees for a number to be regular.
(This result was later given a case-free proof by Lehrer and Michel~\cite{Lehrer-Michel};
see also \cite[Chapter 11.4.2]{Lehrer-Taylor}.)
Consider the action of $W$ on the covariant ring $\Sym(V^*)/J_W$.  Let
$A_i$ be the degree $i$ graded piece of $\Sym(V^*)/J_W$.  Given any
representation $U$ of $W$, we define the {\bf fake degree} of $U$
to be the generating function
\[
  g_U(t)=\sum_{i\geq 0} (\dim_\C \Hom_{\C W} (A_i,U))t^i.
\]
We can write $g_U(t)$ in the form
\[
  g_U(t)=t^{e_1(U)}+\cdots+t^{e_r(U)},
\]
with $e_1(U)\leq e_2(U)\leq\cdots\leq e_r(U)$.
The numbers $e_1(U),\ldots,e_r(U)$ are known as the
{\bf $U$-exponents}.  In the case where $U=V$ is the reflection
representation, these are known simply as the {\bf exponents}, and it is classically known
that $e_i=d_i-1$ and hence that
$\sum_{i=1}^n e_n$ is the degree of $f_W$.
In the case $U=V^*$, the $V^*$-exponents are known as the {\bf coexponents},
and the numbers $d_i^*=e_i(V^*)-1$ are the {\bf codegrees}.

Now the criterion of Lehrer and Springer is as follows.

\begin{thm}[\cite{Lehrer-Springer}]\label{thm: regular number criterion}
A number $d$ is regular if and only if the number of degrees that $d$ divides is
the same as the number of codegrees $d$ divides, or, in notation,
\[
  \# \{ i : d|d_i \} = \# \{ i : d|d^*_i \} .
\]
\end{thm}

Note that $d^*_1=0$, since $A_1$, the degree $1$ piece of
$\Sym(V^*)/J_W$, is $V^*$.
Hence, every regular number divides at least one codegree,
so every regular number divides at least one degree,
and in particular no regular number exceeds the largest degree $d_n$.

Using this criterion, one can easily calculate using the Shephard--Todd classification and
a table of degrees and codegrees (for example in \cite[Appendix D.2]{Lehrer-Taylor}) that the only irreducible
reflection groups whose largest degrees are not regular numbers are the sporadic group $G_{15}$ and the imprimitive groups
$G(de, e, n)$ where $d,e,n\geq2$ and at least one of $e>2$ or $n>2$.

For a reducible reflection group $W = W_1 \times W_2$,
it follows easily from the definitions
that a number is regular for $W$ if and only if it is regular for each of $W_1$ and $W_2$,
while the degrees of $W$ are simply the union (as a multiset) of the degrees of $W_1$ and of $W_2$.
These facts imply that a reducible reflection group has its highest degree regular if and only if
the highest degrees of its irreducible components are equal and this number is regular for each of the components.
Note also that $f_{W_1 \times W_2} = f_{W_1} f_{W_2}$.

\section{Proof of Main Theorem}\label{section: main proof}

Now we give upper bounds for $r(f_W)$ by giving expressions for $f_W$
as sums of powers of linear forms.

For a polynomial $p \in S=\Sym(V)$ the \textbf{skew-symmetrization} of $p$ is
\begin{equation}\label{eqn: alternating sum}
  \alt(p) := \frac{1}{|W|}\sum_{w\in W} (\det w) (w \cdot p).
\end{equation}
Note that for $w \in W$, $w \cdot \alt(p) = (\det w)^{-1} \alt(p)$, so $\alt(p)$ is skew invariant for any $p$.

Recall that if $p\in S$ is a skew invariant polynomial, then $f_W$ divides $p$.
In particular, if $D = \sum_{i=1}^n (d_i-1)$ is the degree of $f_W$,
and $p$ has degree $D$, then $\alt(p)$ must be a (possibly zero) scalar multiple of $f_W$.
Therefore, given any linear form $L\in V$, we see that
$\alt(L^D)$ is a (possibly zero) scalar multiple of $f_W$.

We now show the following lemma.

\begin{lem}
With $D$ as above,
the polynomial $\alt(L^D)$ is zero if and only if $L$ lies on
a reflecting hyperplane in $V$.
\end{lem}



\begin{proof}
Let $\C_{\det^{-1}}$ be the rank one representation of $W$
on which $w \in W$ acts as $(\det w)^{-1}$, where the determinant is given
by the determinant of $w$ in its action on the reflection representation $V$.
Then the skew-symmetrization operator $\alt$ is a $\C W$-linear map
$\alt : \Sym^D(V) \rightarrow \C_{\det^{-1}}$,
where $\C_{\det^{-1}}$ is identified with the span of $f_W$.
We can tensor this map by the determinant representation $\mathbb{C}_{\det}$
to get a map
$\alt\otimes \mathbb{C}_{\det} : \Sym^D(V)\otimes \C_{\det} \rightarrow \mathbb{C}$,
which we can identify with an invariant element of
$\Sym^D(V^*)\otimes\mathbb{C}_{\det^{-1}}$,
or equivalently an element of $\Sym^D(V^*)$ on which $W$ acts by $\det$.

Just as $f_W$ divides
every element of $\Sym(V)$ on which $W$ acts by $\det^{-1}$,
we have a corresponding element $f^*_W\in \Sym(V^*)$ (defined only up to a nonzero scalar multiple),
vanishing precisely on the reflecting hyperplanes $H^*_t\subset V$,
each to order $k_t-1$,
that divides every element of $\Sym(V^*)$ on which $W$ acts by $\det$.

The operator $\alt$ is clearly nonzero, since $\alt(p)=p$ for the
polynomials $p$ on which $w \cdot p = (\det w)^{-1} p$ for all $w\in W$,
in particular for $p = f_W$.
Hence, regarded as a tensor, $\alt$ is a nonzero multiple of $f^*_W$,
so $\alt(L^D)$ is zero if and only if the tensor $f^*_W$ vanishes on $L^D$,
if and only if the polynomial $f^*_W$ vanishes on $L$,
if and only if $L$ lies on a reflecting hyperplane.
\end{proof}

Now suppose $L$ is an eigenvector of $w\in W$ with eigenvalue
$e^\frac{2\pi i}{d}$.
Let $C_w\subset W$ be the subgroup generated by $w$,
and let $W/C_w$ denote the collection of its left cosets.
Note $|C_w|$ is some multiple of $d$ and hence $|W/C_w|\leq |W|/d$. 
Picking any representative $\sigma$ for each coset in $W/C_w$, we have
\begin{equation}\label{eq: general explicit expression}
\begin{split}
  \alt(L^D) &= \frac{1}{|W|} \sum_{\sigma \in W/C_w} (\det \sigma)
    \left( \sigma \cdot \sum_{v\in C_w} (\det v) (v\cdot L^D) \right) \\
  &= \frac{1}{|W|} \sum_{\sigma \in W/C_w} (\det \sigma) 
    \left( \sigma \cdot \sum_{j=0}^{|C_w|-1} (\det w)^{j} e^{2\pi i Dj/d} L^D \right) \\
  &= \frac{1}{|W|} \left( \sum_{j=0}^{|C_w|-1} (\det w)^{j} e^{2\pi i Dj/d} \right)
    \sum_{\sigma \in W/C_w} (\det \sigma) (\sigma \cdot L^D) ,
\end{split}
\end{equation}
which is a sum of at most $|W/C_w|\leq |W|/d$ powers of linear
forms.
Hence we have shown the following proposition.

\begin{prop}\label{prop: upper bound}
Let $W$ be a complex reflection group, $f_W$ its fundamental skew
invariant, and $d$ a regular number for $W$.  Then
\[
  r(f_W) \leq \frac{|W|}{d} .
\]
\end{prop}

Combining Proposition \ref{prop: upper bound} and Corollary \ref{cor: lower bound},
we have shown our main theorem:

\begin{thm}
Let $W$ be a finite complex reflection group acting on $\C^n$,
$d_1\leq\cdots\leq d_n$ the degrees of $W$, and suppose that $d_n$ is a regular number.
Then $cr(f_W) = sr(f_W) = r(f_W) = |W|/d_n$.
\end{thm}

See \cite{Bernardi:2012fk} for additional related quantities that are between the Waring rank and cactus rank,
and therefore equal in the case of this theorem.

Consider the factor $\sum_{j=0}^{|C_w|-1} (\det w)^j e^{2\pi i Dj / d}$
appearing in \eqref{eq: general explicit expression}.
Observe that $\det w$ is a $|C_w|$-th root of unity, as is $e^{2\pi i D/d}$,
since $d$ divides $|C_w|$.
Let $\zeta = (\det w) e^{2 \pi i D / d}$.
We are considering the factor $\sum_{j=0}^{|C_w|-1} \zeta^j$,
which is $|C_w|$ if $\zeta=1$, otherwise $0$.
When $L$ is a regular vector, $\alt(L^D) \neq 0$,
so necessarily $\zeta = 1$ and the sum is $|C_w|$.
Then the explicit expression simplifies to
\begin{equation}\label{eq: simplified general explicit expression}
  \alt(L^D) = \frac{|C_w|}{|W|} \sum_{\sigma \in W/C_w} (\det \sigma) (\sigma \cdot L)^D .
\end{equation}
Hence
\[
  f_W = C \sum_{\sigma \in W/C_w} (\det \sigma) (\sigma \cdot L)^D,
\]
for some nonzero scalar $C$.
Note that $f_W$ is only defined up to a scalar factor.

\section{Examples}\label{section: examples}

\subsection{The Vandermonde determinant}

The classical Vandermonde determinant $\prod_{1 \leq i < j \leq n} (x_i-x_j)$
is the fundamental skew invariant of the action
of $W=S_n$ on $\mathbb{C}^n$ by permuting coordinates.
This is an irreducible real reflection group with reflection hyperplanes
defined by the polynomials $x_i-x_j$ for $i\neq j$;
hence its fundamental skew invariant is the Vandermonde determinant
$f_{S_n}=\prod_{i<j} (x_i-x_j)$, of degree $\binom{n}{2}$.

This action is non-essential as it fixes the line spanned by $(1,\dotsc,1)$.
Passing to an essential action,
say by restricting to the subspace defined by $x_1+\dotsc+x_n=0$,
would correspond to changing $f_{S_n}$ by a substitution
such as replacing $x_n$ with $-x_1-\dotsb-x_{n-1}$.

The degrees (other than $1$) are $2, 3, \ldots, n$,
and since $S_n$ is an irreducible real reflection group, the highest degree $n$ is a regular number.
Explicitly, the linear form
$L=\sum_{j=1}^n e^{2\pi i (j-1)/n}x_j$ is a regular element for which the
$n$-cycle $(12\cdots n)$ acts by multiplication by $e^{2\pi i/n}$.

This implies that the Waring rank of the Vandermonde determinant is
$(n-1)!$, with an explicit expression given by
\begin{equation}\label{eq: Vandermonde explicit}
\begin{split}
  \prod_{i<j} (x_i-x_j) &= C \sum_{\sigma\in S_n/C_n}
      \sgn(\sigma) (\sigma \cdot L)^{\binom{n}{2}} \\
    &= C \sum_{\sigma\in S_n/C_n}
      \sgn(\sigma) \left( \sum_{j=1}^n e^{2\pi i (j-1)/n} x_{\sigma(j)} \right)^{\binom{n}{2}}, \\
\end{split}
\end{equation}
for some scalar $C$, see below,
where $C_n$ is the cyclic subgroup generated by $(12\cdots n)$
and $\sigma$ denotes a choice of coset representatives in $S_n/C_n$.
For example, the sum can be taken over permutations $\sigma$
such that $\sigma(1)=1$.

The cactus rank and smoothable rank are also equal to $(n-1)!$.

By the Alexander--Hirschowitz theorem, the rank of a general form
of degree $\binom{n}{2}$ in $n$ variables is
\[
  \left\lceil \frac{1}{n} \binom{ \binom{n}{2} + n - 1 }{ n-1 } \right\rceil
\]
which is greater than $(n-1)!$ for $n>1$.
Thus the Vandermonde determinant $f_{S_n}$ has less than general rank,
and in fact the ratio of $r(f_{S_n})$ to the general rank rapidly approaches zero
as $n$ goes to infinity.

If we work instead with the essential action of $S_n$ on $\C^{n-1}$,
the rank of the fundamental skew invariant should be compared to the generic rank
of a form in one less variable.
This still grows far more rapidly than $r(f_{S_n})$.

Now we find the scalar $C$ in \eqref{eq: Vandermonde explicit}.
Consider the monomial $x_1^0 x_2^1 \dotsm x_n^{n-1}$.
Its coefficient on the left hand side of \eqref{eq: Vandermonde explicit}
is $(-1)^{\binom{n}{2}}$.
Its coefficient in the power sum on the right hand side is
\[
  \sum_{\sigma \in S_n/C_n} (\sgn \sigma) \binom{\binom{n}{2}}{1,2,\dotsc,n-1}
    \prod_{j=1}^n \left( e^{2\pi i (\sigma^{-1}(j)-1)/n} \right)^{j-1} .
\]
We fix the choice of coset representatives to be the set of $\sigma \in S_n$ such that $\sigma(1) = 1$.
For convenience let $\alpha$ be a primitive $n$th root of unity, say $e^{2\pi i / n}$.
Let $M_n$ be the multinomial coefficient $\binom{\binom{n}{2}}{1,2,\dotsc,n-1}$.
Reindexing by $\sigma^{-1}$ instead of $\sigma$,
we get
\[
\begin{split}
  \frac{1}{C} &= (-1)^{\binom{n}{2}} M_n
    \sum_{\substack{\sigma \in S_n \\ \sigma(1) = 1}} (\sgn \sigma) \prod_{j = 2}^n \alpha^{(j-1)(\sigma(j)-1)} \\
    &= (-1)^{\binom{n}{2}} M_n \det
      \begin{pmatrix}
        \alpha & \alpha^2 & \dots & \alpha^{n-1} \\
        \alpha^2 & \alpha^4 & \dots & \alpha^{2(n-1)} \\
        \vdots & \vdots & & \vdots \\
        \alpha^{n-1} & \alpha^{2(n-1)} & \dots & \alpha^{(n-1)^2}
      \end{pmatrix} \\
    &= (-1)^{\binom{n}{2}} M_n \alpha^{\binom{n}{2}} \prod_{1 \leq j < k \leq n-1} (\alpha^k - \alpha^j) \\
    &= (-1)^{\binom{n}{2} + n + 1} M_n \prod_{1 \leq j < k \leq n-1} (\alpha^k - \alpha^j) .
\end{split}
\]

We now need to calculate $\prod_{1 \leq j < k \leq n-1} (\alpha^k - \alpha^j)$, which we denote by $P$.
We will calculate $|P|$ and $\arg P$ separately.
This calculation has most likely been known since the nineteenth century,
especially the calculation of $|P|^2$, which is (up to sign) the discriminant of the polynomial $x^{n-1}+x^{n-2}+\cdots+x+1$,
but we could not find a simple reference.
Let $V$ be the Vandermonde matrix
$$V=      \begin{pmatrix}
        1 & 1 & 1 & \dots & 1 \\
        \alpha & \alpha^2 & \alpha^3 & \dots & \alpha^{n-1} \\
        \alpha^2 & \alpha^4 & \alpha^6 & \dots & \alpha^{2(n-1)} \\
        \vdots & \vdots & \vdots & & \vdots \\
        \alpha^{n-2} & \alpha^{2(n-2)} & \alpha^{3(n-2} & \dots & \alpha^{(n-1)(n-2)}
      \end{pmatrix},$$
so that $P=\det V$.  Then $|P|^2=\det (VV^\dagger)$, where $V^\dagger$ is the conjugate transpose of $V$.  It is easy to see that
$$VV^\dagger = \begin{pmatrix}
n-1 & -1 & -1 & \dots & -1 \\
-1 & n-1 & -1 & \dots & -1 \\
\vdots & \vdots & \vdots & & \vdots \\
-1 & -1 & -1 & \dots & n-1
\end{pmatrix}.$$
This is the deleted Laplacian of the complete graph on $n$ vertices, so by the Matrix--Tree Theorem, $\det(VV^\dagger)$ counts the number of spanning trees of the complete graph.  Spanning trees of the complete graph are equivalent to trees on $n$ labelled vertices, and there are $n^{n-2}$ such trees.  Hence $|P|=\sqrt{n^{n-2}}$.

To calculate $\arg P$, we sum up $\arg (\alpha^k-\alpha^j)$ for $1\leq j<k\leq n-1$.
Note that $|\alpha^k|=|-\alpha^j|$, so $\arg(\alpha^k-\alpha^j)=(\arg(\alpha^k)+\arg(-\alpha^j))/2$
as long as we are careful to take arguments which are within $\pi$ of each other.
Since $1 \leq j < k \leq n-1$, we have $2\pi k/n - \pi < 2\pi j/n + \pi < 2\pi k/n + \pi$.
So $\arg(-\alpha^j) = 2\pi j/n + \pi$ is within $\pi$ of $\arg(\alpha^k) = 2\pi k/n$.
Hence
\[
  \arg(\alpha^k-\alpha^j)=\frac{(2\pi k/n)+(2\pi j/n)+\pi}{2},
\]
and
\[
\begin{split}
  \arg P &= \sum_{1\leq j<k\leq n-1} \frac{(2\pi k/n)+(2\pi j/n)+\pi}{2}\\
    &=  \pi\left(\sum_{1\leq j<k\leq n-1} \frac{j+k}{n}\right) + \frac{(n-1)(n-2)\pi}{4}\\
    &=  \pi\left(\sum_{\ell=1}^{n-1} \frac{(n-2)\ell}{n}\right) + \frac{(n-1)(n-2)\pi}{4}\\
    &= \frac{3(n-1)(n-2)\pi}{4}.
\end{split}
\]

Hence, $P=(-i)^{\binom{n-1}{2}}\sqrt{n^{n-2}}$, and
$$\frac{1}{C}=i^{\binom{n-1}{2}} M_n \sqrt{n^{n-2}}.$$
Therefore, we have the following expression for the Vandermonde determinant.
\[
  \prod_{i<j} (x_i-x_j) =
    \frac{(-i)^{\binom{n-1}{2}}}{\binom{\binom{n}{2}}{1,2,\ldots,n-1}\sqrt{n^{n-2}}}
                      \sum_{\sigma\in S_n/C_n} \sgn(\sigma)
                      \left( \sum_{j=1}^n e^{2\pi i (j-1)/n} x_{\sigma(j)} \right)^{\binom{n}{2}}.
\]

\subsection{Types B and D}

The group $B_n$ of order $2^nn!$ acts on $\C^n$, and its fundamental
skew invariant is $f_{B_n}=\prod_{i=1}^n x_i \prod_{i<j}
(x_i-x_j)(x_i+x_j)$, of degree $n + 2 \binom{n}{2} = n^2$.
Its highest degree is $2n$.
In the realization of $B_n$ as signed permutations,
one choice of regular element of
order $2n$ is the signed permutation $v\in B_n$ with $v(i)=i+1$ for
$1\leq i<n$ and $v(n)=\overline{1}$, so in its action on $\C^n$,
$v\cdot x_i=x_{i+1}$ for $i<n$ and $v\cdot x_n = -x_1$.  The linear
form $L=\sum_{j=1}^n e^{2\pi i (j-1)/2n} x_j$ is a regular form on
which $v$ acts by multiplication by $e^{2\pi i/2n}$.

Thus the rank $r(f_{B_n}) = 2^{n-1}(n-1)!$, and an explicit expression
is given by
\[
  \prod_{i=1}^n x_i \prod_{i<j} (x_i-x_j)(x_i+x_j)
  = C \sum_{\sigma\in B_n/C_n} 
    \sgn(\sigma) (\sigma \cdot L)^{n^2},
\]
for some scalar $C$,
where $C_n$ is the cyclic subgroup
generated by $v$ and the $\sigma$ are coset representatives for $B_n/C_n$.
For example we may take $\{\sigma \in B_n \mid \sigma(1) = 1\}$.

The group $D_n$ of order $2^{n-1}n!$ acts on $\C^n$ as a real reflection group
with fundamental skew invariant $f_{D_n}=\prod_{i<j} (x_i-x_j)(x_i+x_j)$,
of degree $2 \binom{n}{2} = n(n-1)$.
Its highest degree (assuming $n\geq 3$) is $2n-2$.
Thus $r(f_{D_n}) = 2^{n-2} n (n-2)!$.
One choice of regular element
of order $2n-2$ is the signed permutation $v\in D_n$ with $v(1)=\overline{1}$,
$v(i)=i+1$ for $2\leq i\leq n-1$, and $v(n)=\overline{2}$, so in its action on
$\C^n$, $v\cdot x_1=-x_1$, $v\cdot x_i=x_{i+1}$ when $2\leq i\leq
n-1$, and $v\cdot x_n=-x_2$.
The linear form $L=\sum_{j=2}^n e^{2\pi i (j-2)/(2n-2)} x_j$
is a regular form for which $v$ acts by
multiplication by $e^{2\pi i/(2n-2)}$.
Thus
\[
  \prod_{i<j} (x_i-x_j)(x_i+x_j)
  = C \sum_{\sigma \in D_n/C_n}
    \sgn(\sigma) (\sigma \cdot L)^{n(n-1)},
\]
for some scalar $C$,
where $C_n$ is the cyclic subgroup generated by $v$
and the $\sigma$ are coset representatives for $D_n/C_n$.

As these groups are irreducible real reflection groups,
the cactus, smoothable, and Waring ranks are equal.

By the Alexander--Hirschowitz theorem, the rank of a general form
of degree $n^2$, respectively~$n(n-1)$, in $n$ variables is
\[
  \left\lceil \frac{1}{n} \binom{n^2 + n - 1}{n-1} \right\rceil,
  \qquad
  \text{respectively}
  \qquad
  \left\lceil \frac{1}{n} \binom{n(n-1) + n - 1}{n-1} \right\rceil,
\]
and these are greater than $2^{n-1}(n-1)!$, respectively\ $2^{n-2} n (n-2)!$.
Again $f_{B_n}$ and $f_{D_n}$ have less than general rank,
and ratios of their ranks to the general ranks go to zero.

\subsection{Exceptional groups}

We list the fundamental skew invariant, its degree, and its rank
for the exceptional real reflection groups $E_8$, $E_7$, $E_6$, $F_4$, $H_4$, and $H_3$.
Here, we regard $E_7$ and $E_6$ as acting non-essentially on $\C^8$
(as subgroups of $E_8$).

The fundamental skew invariants of these groups are in Table~\ref{tab:excep_skew_invar}, with $\beta=\cos(\pi/5)$.

\begin{table}[htbp]
\centering
\[
\begin{array}{ll}
W & f_W \\
\toprule
E_8 &
  \displaystyle \prod_{1 \leq i < j \leq 8} (x_i-x_j)(x_i+x_j)
  \prod_{\substack{\lambda_i = \pm 1 \\ \prod \lambda_i = 1 \\ \lambda_1 = 1}} \sum_{i=1}^8 \lambda_i x_i \\
E_7 &
  \displaystyle (x_1+x_8) \prod_{2 \leq i < j \leq 7} (x_i-x_j)(x_i+x_j)
  \prod_{\substack{\lambda_i = \pm 1 \\ \prod \lambda_i = 1 \\ \lambda_1 = \lambda_8 = 1}} \sum_{i=1}^8 \lambda_i x_i \\
E_6 &
  \displaystyle \prod_{3 \leq i < j \leq 7} (x_i-x_j)(x_i+x_j)
  \prod_{\substack{\lambda_i = \pm 1 \\ \prod \lambda_i = 1 \\ \lambda_1 = \lambda_2 = \lambda_8 = 1}} \sum_{i=1}^8 \lambda_i x_i \\
F_4 &
  \displaystyle x_1 x_2 x_3 x_4 \prod_{1 \leq i < j \leq 4} (x_i-x_j)(x_i+x_j)
  \prod_{\substack{\lambda_i = \pm 1 \\ \lambda_1 = 1}} \sum_{i=1}^4 \lambda_i x_i \\
H_3 &
  \displaystyle x_1 x_2 x_3 \prod_{\substack{\sigma \in S_3 \\ \sigma \text{ even} }}
    \prod_{\substack{ \lambda_i = \pm 1 \\ \lambda_1 = 1 }}
    \Big( (\beta+\frac{1}{2}) \lambda_1 x_{\sigma 1} + \lambda_2 \beta x_{\sigma 2} + \lambda_3 \frac{1}{2} x_{\sigma 3} \Big) \\
H_4 &
  \displaystyle x_1 x_2 x_3 x_4
    \prod_{\substack{\sigma \in S_4 \\ \sigma \text{ even} }}
    \prod_{\substack{ \lambda_i = \pm 1 \\ \lambda_1 = 1 }}
    \Big( (\beta+\frac{1}{2}) \lambda_1 x_{\sigma 1} + \lambda_2 \beta x_{\sigma 2} + \lambda_3 \frac{1}{2} x_{\sigma 3} + 0 x_{\sigma 4} \Big)
    \prod_{\substack{\lambda_i = \pm 1 \\ \lambda_1 = 1}} \sum \lambda_i x_i
\end{array}
\]
\caption{Fundamental skew invariants for exceptional groups.}
\label{tab:excep_skew_invar}
\end{table}

In Table~\ref{tab:excep_rank}, we list, for each group, the number of variables, degree, and rank of its fundamental skew invariant,
along with the general rank for forms in that number of variables and of that degree
and the quotient of the rank by the general rank.
\begin{table}[htbp]
\centering
\[
\begin{array}{lrrrrl}
W & \text{variables} & \deg(f_W) & r(f_W) & \text{general rank} & \text{quotient} \\
\toprule
E_8 & 8 & 120 & 64 \cdot 9! & 11169551972 & 2.08 \cdot 10^{-3} \\
E_7 & 8 & 63 & 4 \cdot 8! & 149846840 & 1.08 \cdot 10^{-3} \\
E_6 & 8 & 36 & 6 \cdot 6! & 4028015 & 1.07 \cdot 10^{-3} \\
F_4 & 4 & 24 & 96 & 732 & 0.131 \\
H_4 & 4 & 60 & 480 & 9928 & 0.0242 \\
H_3 & 3 & 15 &  12 &   46 & 0.261
\end{array}
\]
\caption{Comparison of ranks for skew invariants of real exceptional groups with general rank.}
\label{tab:excep_rank}
\end{table}
Since we have written $f_{E_7}$ and $f_{E_6}$ as polynomials in $8$ variables,
we have compared their ranks with the general ranks of forms in that many variables.
One may instead pass to $7$ or $6$ variables respectively.

Finally, briefly,
the dihedral groups $I_2(m)$ (including $G_2 = I_2(6)$) have fundamental skew invariant $f_{I_2(m)} = x^m - y^m$
in a suitable choice of coordinates, whence the Waring rank is clearly $2$.

\subsection{Imprimitive Groups}

We find the following rank of a natural generalization of the Vandermonde determinant, and of the case $D_n$:
\[
  r \left( \prod_{1 \leq i < j \leq n} (x_i^e - x_j^e) \right) = e^{n-2} n (n-2)! \, .
\]
Indeed, this polynomial is the fundamental skew invariant of a certain imprimitive reflection group.
See \cite[Chapter 2]{Lehrer-Taylor} for information on these groups.
These groups are denoted $G(de,e,n)$ for positive integers $d,e,n$.
By \cite[Lemma~2.8]{Lehrer-Taylor}, the fundamental skew invariant of $G(de,e,n)$ is
\[
  \prod_i x_i^{d-1} \prod_{i < j} (x_i^{de} - x_j^{de}) .
\]
This has degree $n(d-1) + \binom{n}{2} de$.
The polynomial displayed above is the fundamental skew invariant of $G(e,e,n)$.
(Note that $A_n = G(1,1,n)$, $B_n = G(2,1,n)$, and $D_n = G(2,2,n)$.)

The degrees of $G(de,e,n)$ are $de, 2de, \dotsc, (n-1)de$, and $nd$.
The largest degree is a regular number except in the cases
$d,e,n \geq 2$ and at least one of $e > 2$ or $n > 2$;
we exclude these cases.
We also exclude the cases $n \leq 2$ since the resulting polynomials are binary
and their ranks are known by Sylvester's results.
So we assume that $n > 2$ and the largest degree is a regular number;
thus $d=1$ or $e \leq 2$.
If $e = 1$ then the greatest degree is $nd$ and it is a regular number.
Hence the Waring rank (and cactus rank and smoothable rank) of the fundamental skew invariant is
\[
  (d)(2d)\dotsm((n-1)d) = d^{n-1} (n-1)! \, .
\]
If $e > 1$ then the greatest degree is $(n-1)de$ and it is a regular number.
Hence the Waring rank (and cactus rank and smoothable rank) of the fundamental skew invariant is
\[
  (de)(2de)\dotsm((n-2)de)(n) = (de)^{n-2} n (n-2)! \, .
\]

\subsection{Monomials}

The monomial $x_1^{a_1}\cdots x_n^{a_n}$, $a_1 \leq \dotsb \leq a_n$, is the fundamental
skew invariant for the reflection group
$\Z/(a_1+1)\Z \times \dotsb \times \Z/(a_n+1)\Z$,
where the generator of the $j$-th component
$\Z/(a_j+1)\Z$ acts on $x_j$ by multiplication by $e^{-2\pi i/(a_j+1)}$.

The degrees for this group are $d_i=a_i+1$.
We immediately recover the result of Ranestad and Schreyer
that the rank, smoothable rank, and cactus rank satisfy
\[
  r(x_1^{a_1} \dotsm x_n^{a_n}) \geq sr(x_1^{a_1} \dotsm x_n^{a_n}) = cr(x_1^{a_1} \dotsm x_n^{a_n}) = (a_1+1)\dotsm(a_{n-1}+1),
\]
see \cite{MR2842085}.
However, all of the codegrees are $0$;
since the representation $V^*$ decomposes into $n$
distinct $1$-dimensional representations,
we have
\[
  \dim_\C \Hom(V^*,\Sym\nolimits^1(V^*)) = \dim_\C \Hom(V^*,V^*) = n ,
\]
so all the coexponents are $1$.
Therefore, according to Theorem \ref{thm: regular number criterion},
the only case where the largest degree $a_n+1$ is a
regular number is when all the $a_i$, and hence all the degrees, are
equal.

Alternatively, note that the regular numbers of $\Z/(a_i+1)\Z$
are just the divisors of $a_i+1$,
and recall the the regular numbers of a reducible reflection group
are precisely the numbers that are regular for each of the irreducible components.
Thus the greatest regular number of $\Z/(a_1+1)\Z \times \dotsb \times \Z/(a_n+1)\Z$
is $\gcd(a_1+1,\dotsc,a_n+1)$.
Again, this is equal to the largest degree $a_n+1$ if and only if all the $a_i$ are equal.

In particular, we recover the observation of Ranestad and Schreyer
that $r((x_1\dotsm x_n)^d) = (d+1)^{n-1}$ \cite{MR2842085}.

For other monomials, the lower bound of Ranestad and Schreyer is not tight,
and in most cases our upper bound also is not tight.
Indeed, as mentioned in the introduction,
Carlini, Catalisano, and Geramita show
that $r(x_1^{a_1} \dotsm x_n^{a_n}) = (a_2+1)\dotsm(a_n+1)$
when $a_1 \leq \dotsb \leq a_n$ \cite{Carlini20125}.
Regarding upper bounds,
our construction in Proposition \ref{prop: upper bound}
gives a tight upper bound only in the case
$a_1+1$ divides $a_j+1$ for all $j$.

See \cite[Section 4.1]{Carlini20125}
and \cite{Holmes:2014hl}
for a comparison of the ranks of monomials
with general ranks.

It is interesting to note that the explicit expression for
$x_1^{a_1} \dotsm x_n^{a_n}$ as a sum of powers
given in \cite{MR3017012} (which is not the only such expression)
is exactly the skew-symmetrization of $(x_1+\dotsb+x_n)^D$,
$D = a_1+\dotsb+a_n$,
over the subgroup $\{1\} \times \Z/(a_2+1)\Z \times \dotsb \times \Z/(a_n+1)\Z$
rather than over the whole group.
Furthermore $x_1+\dotsb+x_n$ is not necessarily an eigenvector of any element in the group,
depending on whether the $a_i+1$ have a common divisor.

\bigskip

The proof of Theorem~\ref{thm: regular number criterion}
by Lehrer and Michel~\cite{Lehrer-Michel}
and the determination of the ranks of monomials in~\cite{Carlini20125,MR3017012}
both involve the algebraic geometry of the covariant ring,
so we speculate that combining these ideas appropriately
could lead to a method to determine the rank of the fundamental skew invariant
for an arbitrary complex reflection group even in the case
where the largest degree is not a regular number.

%

\subsection*{Acknowledgements}
We thank Vic Reiner and Hirotachi Abo for helpful comments.
We also thank the anonymous referees for numerous helpful comments and suggestions
which significantly improved the paper.

\bigskip


\providecommand{\bysame}{\leavevmode\hbox to3em{\hrulefill}\thinspace}
\renewcommand{\MR}[1]{}

\bigskip

\end{document}